\documentclass[amscd,amssymb,verbatim,12pt]{amsart}
\usepackage[all]{xy}
\usepackage{amssymb,latexsym, amsmath, amscd, array, graphicx}
\usepackage{pdfsync}

\swapnumbers \numberwithin{equation}{section}

\theoremstyle{plain}

\newtheorem{thm}{Theorem}[subsection]

\newtheorem{prop}[thm]{Proposition}
\newtheorem{cor}[thm]{Corollary}

\theoremstyle{definition}

 \newcommand{\To}{\longrightarrow}

\DeclareMathOperator{\Int}{{\rm Int}}

\def\Int{\protect\operatorname{Int}}

\def\Z{{\mathbb Z}}

\def\R{{\mathbb R}}

\def\1{\hbox{\rm\rlap {1}\hskip.03in{\rom I}}}
\def\Bbbone{{\rm1\mathchoice{\kern-0.25em}{\kern-0.25em}
{\kern-0.2em}{\kern-0.2em}I}}

\long\def\forget#1\forgotten{} %

\newcommand\ver[1]{\marginpar{\tiny Changed in Ver \VER}}

\date{\today}
\begin{document}

\title[On homology of the complements]{On homology of  complements of compact sets in Hilbert Cube }

\author[A.~Amarasinghe]{Ashwini Amarasinghe}

\author[A.~Dranishnikov]{Alexander  Dranishnikov$^{1}$}

\thanks{$^{1}$Supported by NSF, grant DMS-1304627}

\address{Ashwini Amarasinghe, Department of Mathematics, University
of Florida, 358 Little Hall, Gainesville, FL 32611-8105, USA}
\email{aswini.rc@math.ufl.edu}

\address{Alexander N. Dranishnikov, Department of Mathematics, University
of Florida, 358 Little Hall, Gainesville, FL 32611-8105, USA}
\email{dranish@math.ufl.edu}


\begin{abstract}
We introduce the notion of spaces with weak relative cohomology and  show the acyclicity of the complement $Q\setminus X$ 
in the Hilbert cube $Q$ of a compactum $X$ with weak relative cohomology. As a corollary we obtain the acyclicity of the complement results when

(a) $X$ is weakly infinite dimensional;

(b) $X$ has finite cohomological dimension.

\end{abstract}

\maketitle \tableofcontents

\section {Introduction}
This paper was motivated in part by a recent results of Belegradek and Hu~\cite{BH} about connectedness properties of
the space ${\mathcal R}^k_{\ge 0}(\R^2)$ of non-negative curvature metrics on $\R^2$ where $k$ stands for $C^k$ topology.
In particular, they proved that the complement ${\mathcal R}^k_{\ge 0}(\R^2)\setminus X$ is connected for every finite dimensional
$X$. In~\cite{A} this result was extended to weakly infinite dimensional spaces $X$. The main topological idea of these results is that a 
weakly infinite dimensional compact set cannot separate the Hilbert cube. 

Banakh and Zarichnyi posted in~\cite{P} a problem (Problem Q1053) whether the complement $Q\setminus X$
to a weakly infinite dimensional
compactum $X$ in the Hilbert cube $Q$  is acyclic. They were motivated by the facts that $Q\setminus X$
is acyclic in the case when $X$ is finite dimensional~\cite{K} and when it is countably dimensional
(see~\cite{BCK}). It turns out that the problem was already answered affirmatively by Garity and Wright in the 80s~\cite{GW}. It follows from 
Theorem 4.5~\cite{GW}  which states that finite codimension closed subsets of the Hilbert cube are strongly infinite dimensional.
In this paper we found some sufficient cohomological conditions on compact metric spaces  for the acyclicity of the complement in the Hilbert cube
which give an alternative solution for the  Banakh-Zarichnyi problem. 

We introduced a cohomological
version of the concept of strongly infinite dimensional spaces. We call a spaces which is not cohomologically strongly infinite dimensional
as {\em spaces with weak  relative cohomology}. The main result of the paper is an acyclicity statement for the complement 
of compacta with weak relative cohomology in the Hilbert cube.

\section{Main Theorem}

\subsection{Alexander Duality} We recall that for a compact set $X\subset S^n$ there is a natural isomorphism
$AD_n:H^{n-k-1}(X)\to H_k(S^n\setminus X)$ called the Alexander duality. The naturality means that for a closed subset $Y\subset X$
there is a commutative diagram (1):
$$
\begin{CD}
H^{n-k-1}(X) @>AD_n>> \tilde H_k(S^n\setminus X)\\
@VVV @VVV\\
H^{n-k-1}(Y) @ >AD_n>>\tilde H_k(S^n\setminus Y).
\end{CD}
$$
Here we use the singular homology groups and the \v Cech cohomology groups.
We note  the Alexander Duality commutes with the suspension isomorphism $s$ in the diagram (2):
$$
\begin{CD}
H^{n-k-1}(X) @>AD_n>> \tilde H_k(S^n\setminus X)\\
@VsVV @V\cong VV\\
H^{n-k}(\Sigma X) @ >AD_{n+1}>>\tilde H_k(S^{n+1}\setminus \Sigma X).
\end{CD}
$$
This the diagram can be viewed as a special case of the Spanier-Whitehead duality~\cite{W}.

Since $B^n/\partial B^n=S^n$, the Alexander Duality in the $n$-sphere $S^n$ and its property
can be restated verbatim for the $n$-ball $B^n$ for subspaces $X\subset B^n$  with $X\cap\partial B^n\ne \emptyset$ in terms of relative cohomology groups
$$
H^{n-k-1}(X,X\cap\partial B^n) \stackrel{AD}\To H_k(\Int B^n\setminus X).
$$

Note that for a pointed space $X$ the reduced suspension $\Sigma X$ is the quotient space $(X\times I)/(X\times\partial I\cup x_0\times I)$ and the
suspension isomorphism $s:H^i(X,x_0)\to H^{i+1}(\Sigma X)=H^{i+1}(X\times I,X\times\partial I\cup x_0\times I)$ is defined by a cross product with the fundamental class $\phi\in H^1(I,\partial I)$, $s(\alpha)=\alpha\times\phi$. Note that $\alpha\times\phi=\alpha^*\cup\phi^*$ 
in $H^{i+1}(X\times I,X\times\partial I\cup x_0\times I)$ where $\alpha^*\in H^i(X\times I, x_0\times I)$ is the image of $\alpha$ under the induced homomorphism for the projection $X\times I\to X$ and $\phi^*\in H^1(X\times I,X\times\partial I)$ is the image of $\phi$ under the induced homomorphism defined by the projection $X\times I\to I$.
Thus in view of an isomorphism $H_*(\Int B^n\setminus X)\to H_*(B^n\setminus X)$, the commutative diagram (2) stated for $B^n$ turns into 
the following ($2'$)
$$
\begin{CD}
H^{n-k-1}(X,\partial X) @>AD_n>> \tilde H_k(B^n\setminus X)\\
@V-\cup\phi^*VV @V\cong VV\\
H^{n-k}(X\times I,\partial X\times I\cup X\times\partial I) @ >AD_{n+1}>>\tilde H_k(B^{n+1}\setminus  (X\times I))
\end{CD}
$$
where $\partial X=X\cap\partial I^n$.

\subsection{Strongly infinite dimensional spaces}
First, we recall some classical definitions which are due to Alexandroff~\cite{AP}. 
A map $f:X\to I^n$ to the $n$-dimensional cube is called {\em essential} if it cannot be deformed to $\partial I^n$ through the maps of pairs
$(I^n,f^{-1}(\partial I^n))\to (I^n,\partial I^n)$. A map $f:X\to Q=\prod_{i=1}^\infty I$  to the Hilbert cube is called {\em essential} if the composition $p_n\circ f:X\to I^n$ is essential for every projection $p_n:Q\to I^n$ onto the factor.
A compact space $X$ is called {\em strongly infinite dimensional} if it admits an essential map onto the Hilbert cube. An infinite dimensional space
which is not strongly infinite dimensional is called {\em weakly infinite dimensional}. 

We call a relative cohomology class $\alpha\in H^k(X,A)$ {\em pure relative} if $j^*(\alpha)=0$ for the inclusion of pairs 
homomorphism $j^*:H^k(X,A)\to H^k(X)$. This means that $\alpha=\delta(\beta)$ for $\beta\in \tilde H^{k-1}(A)$. 
In the case when $k=1$ the cohomology class $\beta$ can be represented by a map $f':A\to S^0=\partial I$. Let $f:X\to I$ be a continuous extension of $f'$.
Then $\alpha=f^*(\phi)$ for 
the map $f:(X,A)\to (I,\partial I)$ of pairs where $\phi$ is  the fundamental class.
\begin{prop}\label{pure}
Suppose that for a compact metric space $X$ there  a sequence of pure relative classes $\alpha_i\in H^1(X,A_i)$, $i=1,2,\dots$,
$A_i\subset_{Cl}X$ such that the cup-product $\alpha_1\cup\alpha_n\cup\dots\cup\alpha_n\ne 0$ for all $n$.
Then $X$ is strongly infinite dimensional.
\end{prop}
\begin{proof}
Let $f_i:(X,A_i)\to (I,\partial I)$ be representing
$\alpha_i$ maps. Note that the homomorphism $f_n^*:H^n(I^n,\partial I^n)\to H^n(X,\cup_{i=1}^nA_i)$ 
induced by the map  $\bar f_n=(f_1,f_2,\dots,f_n):X\to I^n$
takes the product of the fundamental classes
$\phi_1\cup\dots\cup\phi_n$ to $\alpha_1\cup\alpha_n\cup\dots\cup\alpha_n\ne 0$. Therefore, $f_n$ is essential. Thus,
the map $$f=(f_1,f_2,\dots):X\to\prod_{i=1}^\infty I=Q$$ is essential.
\end{proof}

We call a space $X$ {\em cohomologically strongly infinite dimensional} if there is a sequence of pure relative classes $\alpha_i\in H^{k_i}(X,A_i)$, $k_i>0$, $i=1,2,\dots$,
$A_i\subset_{Cl}X$ such that the cup-product $\alpha_1\cup\alpha_n\cup\dots\cup\alpha_n\ne 0$ for all $n$.
Compacta which are not cohomologically strongly infinite dimensional will be called  as {\em compacta with weak relative cohomology} or {\em relative cohomology weak compacta}.
One can take any coefficient ring $R$ to define relative cohomology weak compacta with respect to $R$.

Thus, the class of cohomologically weak compacta contains weakly infinite dimensional compacta 
(Proposition~\ref{weak}) as well as all compacta with finite cohomological dimension.

\subsection{Proof of the main theorem}
\begin{prop}\label{proj}
Let $X\subset Q$ be a closed subset and let $a\in H_i(Q\setminus X)$ be a nonzero element for some $i$. Then there exists an $n\in\mathbb{N}$ such that the image $a_n$ of $a$ under the inclusion homomorphism  $$H_i(Q\setminus X)\to H_i(Q\setminus p^{-1}_np_n(X))$$ is nontrivial. 
\end{prop}
\begin{proof}
This follows from the fact that $$H_i(Q\setminus X)=\underset{\rightarrow}{\lim}H_i(Q\setminus p^{-1}_np_n(X)).$$
\end{proof}

\begin{thm}\label{main}
Let $X\subset Q$ be  a  compact subset of the Hilbert cube with weak relative cohomology
with coefficients in a ring with unit $R$. 
Then $\tilde H_i(Q\setminus X;R)=0$ for all $i$.
\end{thm}
\begin{proof}
We assume everywhere below that the coefficients are taken in $R$.

We note that $Q\setminus p^{-1}_np_n(X)=(I^n\setminus p_n(X))\times Q_n$ where $Q_n$ is the Hilbert cube complementary to $I^n$
in $Q=I^n\times Q_n$.

Assume the contrary, $\tilde H_i(Q\setminus X)\ne 0$.
Let $a\in \tilde H_i(Q\setminus X)$ be a nonzero element.
We use the notation $X_k=p_k(X)$. By Proposition~\ref{proj} there is $n$ such that $a_k\ne 0$, $a_k\in \tilde H_i((I^k\setminus X_k)\times Q_k)$ for all $k\ge n$.
Thus, we may assume that $a_k$ lives in $\tilde H_i((I^k\setminus X_k)\times\{0\})\cong \tilde H_i((I^k\setminus X_k)\times Q_k)$. 

Let $\beta_k\in H^{k-i-1}(X_k,X_k\cap\partial I^k)$ be  dual to $a_k$ for the Alexander duality applied in $(I^k,\partial I^k)$.
By the suspension isomorphism for the Alexander duality (see diagram ($2'$)), $a_k$ is dual in $I^{k+1}$ to $\beta_k^*\cup\phi^*$ where $\beta_k^*$ is the image of $\beta_k$ under the homomorphism induced by the projection $X_k\times I\to X_k$ and $\phi^*$ is the image of the fundamental class $\phi\in H^1(I,\partial I)$ under the homomorphism induced by the projection $X_k\times I\to I$. 

By the naturality of the Alexander duality
(we use the $B^n$-version of the diagram (1)) applied to the inclusion $X_{k+1}\subset X_k\times I$ we obtain that $a_{k+1}$ is dual to $\beta_k'\cup\phi_{k+1}$ where
$\beta_{k+1}'$ is the restriction of $\beta^*_k$ to $X_{k+1}$ and $\phi_{k+1}$ is the restriction of $\phi^*$ to $X_{k+1}$.
Note that $\beta'_{k+1}=(q^{k+1}_k)^*(\beta_k)$ where $q^{k+1}_k:X_{k+1}\to X_k$ is the projection $p^{k+1}_k:I^{k+1}\to I^k$ restricted to $X_{k+1}$. Thus,  $\beta_{k+1}=(q^{k+1}_k)^*(\beta_k)\cup\phi_{k+1}$.

By induction on $m$ we obtain
 $$(*)\ \ \ \ \ \beta_{n+m}=(q^{n+m}_n)^*(\beta_n)\cup\psi_1\cup\dots\cup\psi_m$$ 
where $\psi_j=(q^{n+m}_{n+j})^*(\phi_{n+j})$, $j=1,\dots,m$.

We define a sequence $\alpha_j=(q_{n+j})^*(\phi_{n+j})$ of relative cohomology classes on $X$ where $q_n=p_n|_X:X\to X_n$. 
Since $X=\underset{\leftarrow}{\lim} X_r$, for the \v Cech cohomology we have $$H^\ell(X,q^{-1}_{n+k}(A))=\underset{\rightarrow}\lim\{H^\ell(X_{n+m},(q^{n+m}_{n+k})^{-1}(A))\}_{m\ge k}$$ for any $A\subset_{Cl}X_{n+k}$. Hence for fixed $k$
the product $\alpha_1\cup\dots\cup\alpha_k$ is defined  by the thread $$\{(q^{n+m}_{n+1})^*(\phi_{n+1})\cup\dots\cup(q^{n+m}_{n+k})^*(\phi_{n+k})\}_{m\ge k}$$
in the above direct limit.
Since $\beta_{n+m}\ne 0$, in view of ($\ast$) we obtain that
$$(q^{n+m}_{n+1})^*(\phi_{n+1})\cup\dots\cup(q^{n+m}_{n+k})^*(\phi_{n+k})=\psi_1\cup\dots\cup\psi_k\ne 0.$$
Hence, $\alpha_1\cup\dots\cup\alpha_k\ne 0$ for all $k$. Therefore,  $X$ is cohomologically strongly infinite dimensional. This contradicts to the assumption.
\end{proof}
\begin{cor}
Let $X\subset Q$ be a compactum with finite cohomological dimension. Then $\tilde H_i(Q\setminus X)=0$ for all $i$.
\end{cor}
We recall that the cohomological dimension of a space $X$ with integers as the coefficients is defined as $$\dim_{\Z}X=\sup\{n\mid \exists\ A\subset_{Cl}X\ \text{with}\  H^n(X,A)\ne 0\}.$$

\begin{cor}\label{2}
Let $X\subset Q$ be  weakly infinite dimensional compactum. Then $\tilde H_i(Q\setminus X)=0$ for all $i$.
\end{cor}
\begin{proof}
This is rather a corollary of the proof of Theorem~\ref{main} where we constructed a nontrivial infinite product
of pure relative 1-dimensional cohomology classes $\alpha_1\cup\alpha_2\cup\dots .$ on $X$. Then the corollary holds true 
in view of Proposition~\ref{pure}. 
\end{proof}

\section{On weakly infinite dimensional compacta}

Let $\phi\in H^{m-1}(K(G,m-1);G)$ be the fundamental class. Let $\bar\phi$ denote the image $\delta(\phi)\in
H^m(Cone(K(G,k-1)),K(G,k-1);G)$ of the fundamental class under the connecting homomorphism in the exact sequence of  pair.
\begin{prop}\label{rel} 
For any pure relative cohomology class $\alpha\in H^m(X,A;G)$ there is a map of pairs $$f:(X,A)\to (Cone(K(G,m-1)),K(G,m-1))$$ such that $\alpha=f^*(\bar\phi)$.
\end{prop}
\begin{proof}
From the exact sequence of the pair $(X,A)$ it follows that $\alpha=\delta(\beta)$ for some $\beta\in H^{m-1}(A;G)$.
Let $g:A\to K(G,m-1)$ be a map representing $\beta$. Since $Cone(K(G,m-1))$ is an AR, there is an extension
$f:X\to Cone(K(G,m-1))$. Then the commutativity of the diagram formed by $f^*$ and the exact sequence of pairs
implies $\delta g^*(\phi)=f^*\delta(\phi)=f^*(\bar\phi)$.
\end{proof}

Corollary~\ref{2}  can be derived formally from the following:
\begin{thm}\label{weak}
Every weakly infinite dimensional compactum is relative cohomology weak for any coefficient ring $R$.
\end{thm}
\begin{proof}
Let $\alpha_i\in H^{k_i}(X,A_i;R)$ be a sequence of pure relative classes with a nontrivial infinite cup-product. 
We may assume that $k_i>1$. We denote $K_i=K(R,k_i)$.
By Proposition~\ref{rel} there are maps $f_i:(X,A_i)\to (Cone(K_i),K_i)$  representing
$\alpha_i$ in a sense that $f_i^*(\bar\phi_i)=\alpha_i$. Let $$q_i=Cone(c):Cone(K_i)\to Cone(pt)=[0,1]$$ be the cone over the constant map $c:K_i\to \{0\}$.
We show that $g=(q_1f_1,q_2f_2,\dots):X\to \prod_{i=1}^\infty I$ is essential. Assume that for some $n$ the map 
$$g_n=(q_1f_1,q_2f_2,\dots,q_nf_n):X\to \prod_{i=1}^nI=I^n$$ is inessential.
Then there is a deformation $H$ of $g_n$ to a map $h:X\to N_{\epsilon}(\partial I^n)$ rel $g^{-1}_n(N_{\epsilon}(\partial I^n))$
to the $\epsilon$-neighborhood of the boundary $\partial I^n$. Since the map $q_1\times\dots\times q_n$ is a fibration over $\Int I^n$,
there is a lift $\bar H$ of $H$ that deforms $$\bar f_n=(f_1,\dots, f_n):X\to \prod_{i=1}^n Cone(K_i)$$ to an $\epsilon$-neighborhood of
 $M=(q_1\times\dots\times q_n)^{-1}(\partial I^n)$. Since $M$ is an ANR we may assume that $\bar f_n$ can be deformed to $M$
rel $\bar f_n^{-1}(M)$.
Since $\alpha_1\cup\dots\cup\alpha_n=\bar f_n^*(\bar\phi_1\cup\dots\cup\bar\phi_n)$, to complete the proof it suffices to show that  
the restriction of $\bar\phi_1\cup\dots\cup\bar\phi_n$ to $M$ is zero.

We denote $$D=\prod_{i=1}^n Cone(K_i)\setminus \prod_{i=1}^n Ocone(K_i)$$ where
$Ocone(K)$ stands for the open cone.
Since $\prod Cone(K_i)$ is contractible, $\phi_1\cup\dots\cup\phi_n=\delta\omega$ for some $\omega\in H^{k_1+\dots+ k_n-1}(D;R)$.  

We show that $D$ is homeomorphic to $K_1\ast \dots\ast K_n$. Note that 
$$\prod_{i=1}^nCone(K_i)=\{t_1x_1+\dots+ t_nx_n\mid x_i\in K_i,\ t_i\in[0,1]\}$$
and
$$\prod_{i=1}^nOcone(K_i)=\{t_1x_1+\dots+ t_nx_n\mid x_i\in K_i,\ t_i\in[0,1)\}$$
with the convention that $0x=0x'$. Then the set $D$ consists of all $t_1x_1+\dots+ t_nx_n$ such that $t_i=1$ for some $i$. We recall that
$$
K_1\ast \dots\ast K_n=\{t_1x_1+\dots+ t_nx_n\mid x_i\in K_i,\ t_i\in[0,1]\ \sum t_i=1\}$$ with the same convention.
Clearly, the projection of of the unit sphere in $\R^n$ for the $\max\{|x_i|\}$ norm onto the unit sphere for the $|x_1|+\dots+|x_n|$ norm
is bijective.
Thus, the renormalizing map 
$\rho: D\to K_1\ast \dots\ast K_n$, $$\rho(t_1x_1+\dots+t_nx_n)=\frac{t_1}{\sum t_i}x_1+\dots+\frac{t_n}{\sum t_i}x_n$$ is bijective 
and, hence, is a homeomorphism.

The space $M$  can be defined as follows:
$$
M=\{t_1x_1+\dots+ t_nx_n\mid x_i\in K_i,\ t_i\in[0,1]\ \exists i\ : t_i\in{0,1}\}.
$$
Therefore, $M=D\cup L$ where
$$
L=\{t_1x_1+\dots+ t_nx_n\mid x_i\in K_i,\ t_i\in[0,1]\ \exists i\ : t_i=0\}.
$$
We note that $L$ is the cone over 
$$
L^*=\{t_1x_1+\dots+ t_nx_n\in K_1\ast \dots\ast K_n\mid  \exists i\ : t_i=0\}.
$$
Thus, the space $M$ is obtained by attaching to $D$ the cone over $L^*$. 

We show that the inclusion $L^*\to D$ is null-homotopic. Note that $$L^*=\bigcup_{i=1}^nK_1\ast\dots\ast \hat K_i\ast\dots\ast K_n$$ 
where $\hat K_i$ means that the $i$-factor is missing.
Fix points $e_i\in K_i$ and the straight-line deformation $$h_t^i:K_1\ast\dots\ast \hat K_i\ast\dots\ast K_n\to K_1\ast\dots\ast K_n$$ to the point $e_i$.
We recall that the reduced joint product is defined as follows
$$
K_1\tilde\ast\cdots\tilde\ast K_n=\{t_1x_1+\dots+ t_nx_n\mid x_i\in K_i,\ t_i\in[0,1],\ \sum t_i=1\}
$$
with convention that $0x=0x'$ and $e_i=e_j$. Let  $q:K_1\ast\dots\ast K_n\to K_1\tilde\ast\cdots\tilde\ast K_n$ be the
quotient map. Note that $q$ has contractible fibers and hence is a homotopy equivalence.
Then $\tilde x_0=q(\Delta^{n-1})$ is the base point in the reduced joint product   where the simplex $\Delta^{n-1}$ is spanned by $e_1,\dots, e_n$.

The deformation $h^i_t$ defines a deformation $\tilde h_t^i$ of  $K_1\tilde\ast\cdots\tilde\ast \hat K_i\tilde\ast\cdots\tilde\ast K_n$ to the base point 
$\sum t_ie_i$ in the reduced joint product $K_1\tilde\ast\cdots\tilde\ast K_n$. We note that for any $i<j$ the deformations $\tilde h_t^i$  and $\tilde h_t^j$ agree on the common part
$K_1\tilde\ast\cdots\tilde\ast \hat K_i\tilde\ast\cdots\tilde\ast \hat K_j\tilde\ast\cdots\tilde\ast K_n$. Therefore, the union $(\cup\tilde h^i_t)\circ q|_{L^*}$ is 
a well-defined  deformation
of $L^*$ to the base point in the reduced joint product where $q:K_1\ast\dots\ast K_n\to K_1\tilde\ast\cdots\tilde\ast K_n$ is quotient map.
Since $q$ is a homotopy equivalence, $L^*$ is null-homotopic in $K_1\ast\dots\ast K_n$.
\end{proof}
Since there exist strongly infinite dimensional compacta with finite cohomological dimension (see~\cite{D},\cite{DW}) the converse to Theorem~\ref{weak} does not hold true.

The following can be easily derived from from the definition of the cup product for the singular cohomology and the 
definition of the \v Cech cohomology.
\begin{prop}\label{product}
Let $A,A'\subset Y$, $Y,B,B'\subset X$ be closed subsets with $B\cap Y=A$ and $B'\cap Y=A'$.
Then for any ring $R$ there is a commutative diagram generated by inclusions and the cup product
$$
\begin{CD}
H^k(X,B;R)\times H^l(X,B';R) @>\cup>> H^{k+l}(X,B\cup B';R)\\
@VVV @VVV\\
H^k(Y,A;R)\times H^l(Y,A';R) @>\cup>> H^{k+l}(Y,A\cup A';R).
\end{CD}
$$
\end{prop}

\begin{thm}
Let $Y$ be a closed subset of a compactum $X$ with a weak relative cohomology over a ring $R$. Then $Y$ has 
a weak relative cohomology over $R$.
\end{thm}
\begin{proof}
Assume that $Y$ is cohomologically strongly infinite dimensional. Let $\alpha_i\in H^{k_i}(Y,A_i;R)$ be pure relative cohomology classes with nonzero product. Let $$f_i:(Y,A_i)\to (Cone(K(R,k_i-1)),K(R,k_i-1))$$ be a map 
from Proposition~\ref{rel} representing $\alpha_i$. Let $$g_i:X\to Cone(K(R,k_i-1))$$ be an extension of $f_i$
and let $B_i=g_i^{-1}(K(R,k_i-1))$. We consider $\beta_i=g_i^*(\bar\phi_i)\in H^{k_i}(X,B_i)$. Note that
$\alpha_i=\xi_i^*(\beta_i)$ where $\xi_i:(Y,A_i)\to (X,B_i)$ is the inclusion.
By induction, Proposition~\ref{product}, and the fact $\alpha_i\cup\dots\cup\alpha_n\ne 0$ 
we obtain that $\beta_1\cup\dots\cup\beta_n\ne 0$ for all $n$. Then we can conclude that $X$ is cohomologically strongly infinite dimensional.
\end{proof}

\section{Applications}
Let $\Sigma$ denote the linear span of the standard Hilbert cube $Q$ in the Hilbert space $\ell_2$ and let $\Sigma^\omega$ denote the product
of countably many copies of $\Sigma$.
The following is well-known:
\begin{prop}\label{known}
Let $H$ be either $\ell_2$ or $\Sigma^\omega$. Then for any closed subset $A\subset Q$ of the Hilbert cube any embedding $\phi:A\to H$ can be extended to an embedding
$\bar\phi:Q\to H$.
\end{prop}

Let ${\mathcal R}^\gamma(M)$ denote the space of all Riemannian metrics on a  manifold $M$ with the $C^\gamma$ topology.
For the definition of such topology for finite $\gamma$ which are not integers we refer to~\cite{BB}. The following theorem was proved in~\cite{BB}:
\begin{thm}
For a connected manifold $M$ the space ${\mathcal R}^\gamma(M)$ is homeomorphic to $\Sigma^\omega$ when $\gamma<\infty$ and
${\mathcal R}^\infty(M)$ is homeomorphic to $\ell_2$.

The space ${\mathcal R}^\gamma_{\ge 0}(\mathbb R^2)$ is homeomorphic to $\Sigma^\omega$ when $\gamma<\infty$ and it is not an integer.
The space ${\mathcal R}_{\ge 0}^\infty(\mathbb R^2)$ is homeomorphic to $\ell_2$.
\end{thm}
This theorem with our main result implies the following
\begin{cor}
Let $X\subset {\mathcal R}^\gamma(M)$ be a closed weakly infinite dimensional subset
for a connected manifold $M$. Then ${\mathcal R}^\gamma(M)\setminus X$ is acyclic.

Let $X\subset {\mathcal R}_{\ge 0}^\gamma(\R^2)$ be a closed weakly infinite dimensional subset
when $\gamma=\infty$ or $\gamma<\infty$ and it is not an integer.
Then the complement ${\mathcal R}_{\ge 0}^\gamma(\R^2)\setminus X$ is acyclic.
\end{cor}
\begin{proof} In either of the cases we denote the space of metrics by $H$.
Let $f:K\to H\setminus X$ be a singular cycle. By Proposition~\ref{known} there is an embedding of the Hilbert cube $Q\subset H$ such that
$f(K)\subset Q$. By Theorem~\ref{main} $f$ is homologous to zero in $Q\setminus X$ and, hence, in $H\setminus X$.
\end{proof}

\end{document}